\documentclass[11pt]{article} % use larger type; default would be 10pt

\usepackage[utf8]{inputenc} % set input encoding (not needed with XeLaTeX)

\usepackage{amsthm}

\newtheorem{theorem}{Theorem}[section]

\newtheorem{proposition}[theorem]{Proposition}

\usepackage{geometry} % to change the page dimensions
\usepackage{graphicx} % support the \includegraphics command and options

\usepackage{amsmath}
\usepackage{amssymb}

\usepackage{bm}

\usepackage[parfill]{parskip} % Activate to begin paragraphs with an empty line rather than an indent

\usepackage{booktabs} % for much better looking tables
\usepackage{array} % for better arrays (eg matrices) in maths
\usepackage{paralist} % very flexible & customisable lists (eg. enumerate/itemize, etc.)
\usepackage{verbatim} % adds environment for commenting out blocks of text & for better verbatim
\usepackage{subfig} % make it possible to include more than one captioned figure/table in a single float
\usepackage{fancyhdr} % This should be set AFTER setting up the page geometry
\usepackage{sectsty}
\allsectionsfont{\sffamily\mdseries\upshape} % (See the fntguide.pdf for font help)
\usepackage[nottoc,notlof,notlot]{tocbibind} % Put the bibliography in the ToC
\usepackage[titles,subfigure]{tocloft} % Alter the style of the Table of Contents

\usepackage{spverbatim}
\usepackage{listings}

\usepackage{caption}  % Provides caption customization

\title{Strategies in a mis\`{e}re two-player tree search game}
\author{Ben Andrews \\ \small University of Sheffield \\ \small benand34@gmail.com \\ \small 0009-0000-7692-3117}
\begin{document}
\maketitle

\begin{abstract}
In this paper, we analyse a mis\`{e}re tree searching game, where players take turns to guess vertices in a tree with a secret `poisoned' vertex. After each turn, the guessed vertex is removed from the tree and the game continues on the component containing the poisoned vertex, and as soon as a player guesses the poisoned vertex, they lose. We describe and prove the solution when the game is played on a path graph, both between two optimal players and between a player who makes their decisions uniformly at random and an opponent who plays to exploit this. We show that, with two perfect players, the solution involves different guessing strategies depending on the value of $n$ modulo $4.$ We then show that, with a random and an exploitative player, the probability that the exploitative player wins approaches a constant (approximately $0.599)$ as $n$ increases, and that the vertices one away from the leaves of the path are always optimal guesses for them. We also solve the game played on a star graph, and briefly discuss the possibility for extending the analysis to more general trees.
\end{abstract}


\textbf{Acknowledgements.} The author would like to thank the \emph{Applied Probability Trust} for funding their postgraduate studies and thus making this research possible. The author would also like to thank Jonathan Jordan for his supervision and continued support and advice on this work. The author would also like to thank James Cranch for his interest in this project and his proof reading of and suggestions for this paper, and for suggesting and running a test using the SageMath algebraic dependency method \cite{sage} to find that $a$ is likely not the solution of a polynomial. Additionally, the author would like to thank Ravi Boppana, for bringing the tree search game to their attention by email.

\section{Overview and main results}

\subsection{Introduction and definitions}

Our aim is to study a mis\`{e}re tree search game, a played-to-lose variant of the competitive tree searching game of Boppana and Lewis \cite{treesearch}. The game is played between two players on a starting tree $T,$ where one of the vertices is `poisoned'; neither player knows where it is. The players take turns to guess vertices of the tree. If a player guesses the poisoned vertex, they lose, and otherwise, the vertex they guessed is cut from the tree, splitting it into several components. It will then be revealed which component contains the poisoned vertex, and the game will continue on this component. The game is reminiscent of the random process known as `cutting down trees,' introduced by Meir and Moon in 1970 \cite{meir-moon} and an active area of research (see, for example, \cite{markov-chainsaw} and \cite{random-cuttings}).

In Boppana and Lewis's search game, the players were aiming to find the hidden `target' vertex. They described the winning probabilities and best strategies for two optimal players for any tree. They also considered an alternative pair of strategy profiles, where one player plays uniformly at random, and the other plays optimally with this in mind, trying to maximise their winning probability by exploiting the behaviour of the random player. They described the winning probability and the best strategy for the exploitative player in the cases of path graphs and stars. They also considered the case with two random players. In this case, since the players have no strategy, nothing changes when the mis\`{e}re variant is played, except that the definition and probabilities of a win and a loss are swapped, so we will not conduct a new analysis of this case. They suggested the mis\`{e}re model, where players play to avoid guessing the secret vertex, as an extension to their model for future research. 

In this paper, we will analyse the mis\`{e}re game, both with two optimal players and with a random player, known as $R,$ and an exploitative player, known as $X.$ We notate the game played on a general tree $T$ with player $1$ playing first and player $2$ playing second as $MS(T).$ We will focus on the case of a path graph for most of the paper, but we will also solve the case with a star graph in section $4$, and discuss the possibility for extension of the analysis to a general tree in section $5.$

In the case of a path graph, the game can be thought of as being played on the integers from $1$ to $n,$ where players try to avoid guessing the secret number, which we will refer to as the `mine.' This can be thought of as a competitive variant of the higher-or-lower children's guessing game. This game can also be seen online on Novel Games \cite{novelgames}. In sections $2$ and $3,$ where we analyse path graphs, we will use the language of this guessing game played on the integers.

When playing on a path graph, note that there is a symmetry over the range $[1,n].$ Guessing $1$ is equivalent in value to guessing $n,$ as these numbers are both on the extreme ends of the range and reduce it in the same way (to a range of size $n - 1$). Further, guessing $2$ is equivalent in value to guessing $n - 1,$ as they both have the same possiblities for how they divide the range (either into a range of a single number, or of $n - 2$ numbers). This symmetry continues all the way to the numbers in the middle; in general, guessing a number $m$ has the same success probability and consequences as guessing $n+1-m.$

Additionally, note that after each turn, the game reduces to simpler versions of itself (with lower $n)$ after each turn. The new $n$ is the amount of candidate numbers that remain, which all have equal probabilities of being the mine, and the second player, whose turn it is now, plays as the first player in the game with this lower $n.$ 

\subsection{Main results}

Let $P_n$ be the path graph with $n$ vertices.
The behaviour when both players play optimally on a path graph is described in full by the following theorem.

\begin{theorem}
Assuming that both players play optimally, and know this about each other, the probability that player $1$ wins in $MS(P_n)$ is given by
\begin{equation}
	\label{perfect_1}
	1/2\textrm{ when n is even;}
\end{equation}
\begin{equation}
	\label{perfect_2}
	\frac{2k}{4k+1}\textrm{ when }n = 4k+1 \textrm{ for some integer k;}
\end{equation}
\begin{equation}
	\label{perfect_3}
	\frac{2k+2}{4k+3}\textrm{ when }n = 4k+3\textrm{ for some integer k.}
\end{equation}

Further, the guesses that achieve these maximum winning probabilities are described as follows.

\begin{enumerate}
\item When $n = 4k$ for some $k,$ player 1 should choose a number that divides the game into ranges of sizes $4k_1 + 1$ and $4k_2 + 2,$ for any choice of $k_1, k_2$ with $k_1 + k_2 = k - 1.$
\item When $n = 4k + 1$ for some $k,$ all of player 1's choices are of equal strength.
\item When $n = 4k + 2$ for some $k,$ player 1 should choose a number that divides the game into ranges of sizes $4k_1 + 1$ and $4k_2,$ for any choice of $k_1, k_2$ with $k_1 + k_2 = k.$
\item When $n = 4k + 3$ for some $k,$ player 1 should choose a number that divides the game into ranges of sizes $4k_1 + 1$ and $4k_2 + 1,$ for any choice of $k_1, k_2$ with $k_1 + k_2 = k.$
\end{enumerate}
\end{theorem}

This is proven in section 2. The proof uses the method of strong induction, with four distinct classes of cases representing the different remainders of $n$ when divided by four.

We will also describe the solution with a random and an exploitative player. Intuitively, guessing $2$ or $n-1$ is a logical strategy for $X$ because it always gives them a $1/n$ probability of winning immediately, when the mine is the extreme number on the side of their guess, and if the game otherwise continues it remains of large size, allowing more possiblities for $R$ to make theoretical mistakes as the game goes on. Under the assumption that $X$ does play this strategy, let $p_n$ be the probability that $X$ wins $MS(P_n)$ as the first player, and let $q_n$ be the probability that they win $MS(P_n)$ if they are the second player. Then, define $s_n = \sum_{i=1}^n ip_i,$ and finally, define $a_n$ such that \linebreak
$s_n = a_n(n^2 + 7n + 6) - 2(n+1),$ with $a = \lim_{n \to \infty} a_n.$

Define $P(T)$ and $Q(T)$ to be the probabilities that $X$ wins $MS(T)$ against $R$ when they are the first player and the second player respectively.

\begin{theorem}
All possible optimal strategy profiles for $X$ have them always guessing one of the numbers $2$ or $n-1$ in $MS(P_n),$ and $$\lim_{n \to \infty} P(P_n) = \lim_{n \to \infty} Q(P_n) = 2a.$$ 
The value of $2a$ to ten decimal places is $0.5988890438.$
\end{theorem}

We will prove this in section 3, by calculating the limiting wining probability under our assumption about $X$'s strategy, and then prove by induction that this is indeed their best strategy.

When we consider stars in section $4,$ we will also prove that stars are the best possible tree structure on $n$ vertices for the exploitative player, as described by Theorem 1.3. This result is the mis\`{e}re game's equivalent to Theorem 4.10 of \cite{treesearch}.

\begin{theorem} For any tree $T$ with $n$ vertices,
$$P(T) \leq P(S_n) \text{ and } Q(T) \leq Q(S_n).$$
\end{theorem}

\section{Two optimal players on a path graph}

We will prove Theorem 1.1 by strong induction on four base cases $n \leq 4,$ with four families of cases emerging in the inductive step.

\begin{proof} [Proof of Theorem 1.1.] First, we will check the base cases $n = 1,2,3,4,$ and then we will show that if the theorem is true for all numbers up to and equal to $m,$ for any $m \geq 4,$ then it is also true for $m+1.$ This will confirm that the theorem is true for all $n.$

Let $P(n)$ be the proposition that the conditions of Theorem 1.1 apply for $n.$ In our base cases, we will establish that $P(n)$ is true for all $n \leq 4.$

\textbf{Base cases.} When $n=1,$ player 1's probability of winning is $0$ (and all strategies are equal in strength, as there is only one choice) as required. When $n=2,$ player 1 will win half the time (when they avoid guessing the mine), and so their winning probability is $1/2$ as required. The strategy is also as required with $k_1 = k_2 = 0$; any move would split the game into ranges of sizes $0$ and $1.$

When $n=3,$ if player 1 guesses $2,$ they will win in the cases that the mine is $1$ or $3,$ as this will be the only candidate left for player 2 to guess; they will lose only when the mine is $2.$ Hence, they win with a probability of $\frac{2}{3}$ as required. Guessing $1$ or $3$ would be worse, and would give player 1 a winning probability of $\frac{1}{3}$ only. The correct strategy is as required with  $k_1 = k_2 = 0;$ the best move splits the game into two smaller ranges of size $1.$

In the $n=4$ case, there are two essentially distinct options for player 1 - guessing an extreme number ($1$ or $4$) or a middling number ($2$ or $3$). If they choose to guess an extreme number, then, assuming they do not lose immediately by guessing the mine (probability $3/4$), they will leave player $2$ with a range of three numbers to guess from, with certainty. Player 2 will use the strategy from the $n=3$ case of guessing the middle number, and player 1 win with a probability of $1/3$ from here. Overall, player 1 will win with a probability of $\frac{3}{4} \cdot \frac{1}{3} = \frac{1}{4}$ only.

Instead, player 1 does better by choosing a middling number. If they choose $2$ for example, there are three possible outcomes:
\begin{enumerate}
\item They lose immediately by guessing the mine (probability $\frac{1}{4}$);
\item The mine is $1,$ and player 1 wins with certainty, leaving only the number $1$ for player 2 to guess (probability $\frac{1}{4}$);
\item The mine is greater than $2$ (this case happens with probability $\frac{1}{2}$). The game reduces to an $n=2$ case, and both players win from here with probability $\frac{1}{2}$  .
\end{enumerate}

The total probability of a win for player 1 is $\frac{1}{4} + \frac{1}{2} \cdot \frac{1}{2} = \frac{1}{2}$ (summing over wins from cases 2 and 3 in the above list). This is the best strategy possible for player 1, and matches the described strategy with $k_1 = k_2 = 0$, splitting the game into ranges of sizes $1$ and $2$.

\textbf{Inductive step.} We need to show that knowing that $P(n)$ is true for all $n \leq m$ is sufficient to prove that $P(m+1)$ is also true. Each of the four cases $n = 4k, n = 4k+1, n = 4k+2, n = 4k+3$ can be dealt with separately, with similar methodology for each, by examining the different ways that player 1 can choose how many numbers lie either side of their guess, splitting the range of numbers into two smaller ranges. The sizes of these smaller ranges, modulo $4,$ will be the relevant factor. We will detail the proof for the case where $n = 4k + 3$ for some integer $k$ below.

When $n = 4k + 3,$ player $1$ has three choices about the sizes mod $4.$ In the case where they split the game into sizes $4k_1$ and $4k_2+2$, with $k_1 + k_2 = k,$ the relevant outcomes are:

\begin{enumerate}
\item The game is reduced to a range of size $4k_1,$ which happens with probability $\frac{4k_1}{4k + 3}.$ By the inductive hypothesis \eqref{perfect_1}, the probability that each player will win from here is $\frac{1}{2}.$
\item The game is reduced to a range of size $4k_2 + 2,$ which happens with probability $\frac{4k_2 + 2}{4k + 3}.$ By the inductive hypothesis \eqref{perfect_1}, the probability that each player will win from here is $\frac{1}{2}.$
\end{enumerate}

The probability that player $1$ wins is given by $\frac{1}{2} \cdot \frac{4k_1 + 4k_2 + 2}{4k+3} = \frac{2k + 1}{4k + 3}.$

If instead, player $1$ divides the game into sizes $4k_1 + 1$ and $4k_2 + 1,$ with $k_1 + k_2 = k,$ the relevant outcomes are:

\begin{enumerate}
\item The game is reduced to a range of size $4k_1 + 1,$ which happens with probability $\frac{4k_1 + 1}{4k + 3}.$ By the inductive hypothesis \eqref{perfect_2}, the probability that player $1$ will win from here is $\frac{2k_1 + 1}{4k_1 + 1}.$
\item The game is reduced to a range of size $4k_2 + 1,$ which happens with probability $\frac{4k_2 + 1}{4k + 3}.$ By the inductive hypothesis \eqref{perfect_2}, the probability that player $1$ will win from here is $\frac{2k_2 + 1}{4k_2 + 1}.$
\end{enumerate}

The probability that player $1$ wins is given by $\frac{4k_1 + 1}{4k + 3} \cdot \frac{2k_1 + 1}{4k_1 + 1} + \frac{4k_2 + 1}{4k + 3} \cdot \frac{2k_2 + 1}{4k_2 + 1} = .\frac{2(k_1 + k_2) + 2}{4k + 3} = \frac{2k + 2}{4k + 3}.$

If instead, player $1$ divides the game into sizes $4k_1 + 3$ and $4k_2 + 3,$ with $k_1 + k_2 = k-1,$ the relevant outcomes are:

\begin{enumerate}
\item The game is reduced to a range of size $4k_1 + 3,$ which happens with probability $\frac{4k_1 + 3}{4k + 3}.$ By the inductive hypothesis \eqref{perfect_3}, the probability that player $1$ will win from here is $\frac{2k_1 + 1}{4k_1 + 3}.$
\item The game is reduced to a range of size $4k_2 + 3,$ which happens with probability $\frac{4k_2 + 3}{4k + 3}.$ By the inductive hypothesis \eqref{perfect_3}, the probability that player $1$ will win from here is $\frac{2k_2 + 1}{4k_2 + 3}.$
\end{enumerate}

The probability that player $1$ wins is given by $\frac{4k_1 + 3}{4k + 3} \cdot \frac{2k_1 + 1}{4k_1 + 3} + \frac{4k_2 + 3}{4k + 3} \cdot \frac{2k_2 + 1}{4k_2 + 3} = \frac{2(k_1 + k_2) + 2}{4k+3} = \frac{2k - 2}{4k + 3}.$

The strategy with the best outcome for player $1$ is as described in Theorem 1.1, and wins with probability $\frac{2k + 2}{4k + 3},$ as required. The proofs for the other cases, $n = 4k, 4k+1, 4k+2,$ are very similar, following the same method.

This concludes the inductive step; $P(n)$ is true for all positive integers $n.$

\end{proof}

\section{Random player versus exploitative player on a path graph}

We now move on to consider other strategy profiles where the players are non-optimal; specifically, we will focus on the consequences of one player making all their guesses uniformly at random, who we name $R,$ or the random player. The other player, named $X,$ or the exploitative player, knows this and will play for the highest winrate with this in mind. We will prove our main result, Theorem 1.2, which describes an optimal strategy for $X$ and their resulting winrate as $n \to \infty.$

We will start with a suggested strategy for the exploitative player, which is to always guess the number $2$ (or $n-1$, which is strategically equivalent due to the symmetry of the game). Later on, we will be proving that this strategy is optimal by induction. Recall that we define $p_n$ to be the probability that $X$ wins when they play first in $MS(P_n)$ and uses this strategy. We will also define $q_n$ to be the probability that $X$ wins in $MS(P_n)$ with this strategy, but when the random player $R$ is playing first. We also define a sum function, $s_n,$ by $s_n = \sum_{i=1}^n ip_i.$ We have initial conditions on $s_n,$ by analysis of small cases of $MS(P_n).$ Firstly, see that $X$ will always lose $MS(P_1)$ if they play first, so $p_1 = 0$ and $s_1 = 0.$ If $X$ plays first in $MS(P_2),$ they must try to randomly avoid guessing the mine, and they succeed half of the time, so $p_2 = \frac{1}{2}$ and $s_2 = 2p_2 = 1.$ In $MS(P_3),$ if $X$ is playing first, they achieve a winrate of $\frac{2}{3}$ by guessing $2,$ so $p_3 = \frac{2}{3}$ and $s_3 = 2p_2 + 3p_3 = 3.$

We will find a recurrence relation for the higher values of $s_n,$ and use it to find estimates for the values in the sequences $(p_n)$ and $(q_n),$ with a suitable bounded error, to show that these sequences are in fact increasing as long as $n \geq 9.$ From there, we will confirm by induction that our proposed strategy for $X$ is indeed the correct strategy for them, and the limits of the sequences $(p_n)$ and $(q_n).$ 

\subsection{Results}

\begin{proposition} We have the recurrence 
\begin{equation}
	\label{recurrence}
	s_n = s_{n-1} + \frac{2}{n-2} s_{n-3} + 2.
\end{equation}
\end{proposition}

\begin{proposition} For any real value of $a^*,$ the expression $S(n) = a^*n^2 + (7a^* - 2)n + (6a^* - 2)$ is a quadratic solution of \eqref{recurrence}.
\end{proposition}

Let $a_k$ be defined such that $s_k = a_kn^2 + (7a_k - 2)n + (6a_k - 2)  = a_k(k^2 + 7k + 6) - 2(k+1).$ Then, let $b_k$ be the maximum absolute difference between any two of $a_k, a_{k+1}$ and $a_{k+2}.$ We will see in the proof of Proposition 3.3 that $(b_k)$ is a (non-strictly) decreasing sequence, and that all the values of $a_i$ for $i > k+2$ will lie between the smallest and largest of $a_k, a_{k+1}$ and $a_{k+2}.$ Thus, we will find that the sequence $(a_k)$ will have a limit $a,$ and the values of $a_k$ will function as estimates for $a,$ with $b_k$ being an upper bound for all $|a - a_i|$ with $i > k+2.$ Proposition 3.3 also describes the rapidly decreasing nature of the sequence $(b_k);$ at each two steps the value decreases in the order of $\frac{1}{k},$ leading to long-term decreasing of the order $\frac{1}{(2k)!}$ over time.

\begin{proposition} We have that $b_{k-1} \leq \frac{4}{k+1}b_{k-3},$ and that a limit exists for the sequence $(a_k).$ It follows that

\begin{align*}
b_k \leq
\left\{
    \begin {aligned}
         & \frac{2^l}{21(l+1)!} \quad & k=2l \\
         & \frac{ \frac{15}{252} 4^m  }{(2m+3)!!} & k = 2m+1            
    \end{aligned}
\right.
\end{align*}

where $x!!$ is the double factorial function (the product of $x$ and only every other positive integer less than $x;$ the positive integers with the same parity as $x.$)
\end{proposition}

\begin{proposition} For $n \geq 9,$ $p_n$ and $q_n$ are both increasing.
\end{proposition}

\begin{proposition} An optimal strategy for player $X$ is to choose $2$ or $n-1,$ for any $n.$
\end{proposition}

\begin{proposition} The limits as $n \to \infty$ of the sequences $(p_n)$ and $(q_n)$ are both $2a.$
\end{proposition}

In combination, Propositions 3.5 and 3.6 prove Theorem 1.2. By calculating the values of $a_k$ for arbitrarily large $k$ by computer, $a$ can be calculated to any number of decimal places; see the casework in the appendix for a calculation example. The limiting winning probability for $X$ found in Proposition 3.6, given to six decimal places, is $0.598889.$

\subsection{Proofs of Propositions 3.1 to 3.6, and Theorem 1.2}

\begin{proof} [Proof of Proposition 3.1.] When player $X$ guesses $2$ in $MS(P_n),$ they win immediately with probability $\frac{1}{n},$ when the mine is 1, as player $R$ will have to guess $1$ on the next turn. Otherwise, as long as the mine is not 2, the game continues in an interval of width $n-2$ with $R$ to play; this happens with probability $\frac{n-2}{n}$. Hence, $p_n = \frac{1}{n} + \frac{n-2}{n}q_{n-2}$ for $n \geq 2.$

To write an expression for $q_n,$ consider the possibilities when $R$ plays; if they guess a number $i,$ then with probability $\frac{i-1}{n},$ the mine is lower than $i$, and the game continues with width $i-1,$ with $X$ winning with probability $p_{i-1}.$ Similarly, with probability $\frac{n-i}{n},$ the mine is greater than $i,$ and the game continues with width $n-i.$ Finally, $R$ guesses the mine with probability $\frac{1}{n},$ and $X$ wins in that case.

Summing over these outcomes for all choices of $i$ gives that

$$q_n = \frac{1}{n} \sum_{i=1}^{n} \left( \frac{i-1}{n} p_{i-1} + \frac{n-i}{n} p_{n-i} + \frac{1}{n}  \right) $$
for all $n \in \mathbb{N}$.

Then, $$q_{n-2} = \frac{1}{n-2} + \frac{1}{(n-2)^2}  \sum_{i=1}^{n-2} \left( (i-1)  p_{i-1} + (n-2-i) p_{n-2-i}   \right) = \frac{1}{n-2} + \frac{2}{(n-2)^2}  s_{n-3}.$$ Substituting this into $p_n$ gives that $$p_n = \frac{1}{n} + \frac{n-2}{n} \left( \frac{1}{n-2} + \frac{2}{(n-2)^2}  s_{n-3} \right) = \frac{2}{n} + \frac{2}{n(n-2)} s_{n-3} . $$ Multiplying by $n$ gives $n p_n = 2 + \frac{2}{n-2}s_{n-3}.$ Since $s_n - s_{n-1} = n p_n,$ we have $s_n = s_{n-1} + \frac{2}{n-2}s_{n-3} + 2, $ as required.

\end{proof}

\begin{proof} [Proof of Proposition 3.2.] When $S^*(n) = a^*n^2 + (7a^* - 2)n + (6a^* - 2),$ we have that $S^*(n-1) = a^*(n^2 - 2n + 1) + (7a^* - 2)(n - 1) + (6a^* - 2)$ and $S^*(n-3) = a^*(n^2 - 6n + 9) + (7a^* - 2)(n - 3) + (6a^* - 2).$ Substituting all of this into \eqref{recurrence} gives, for the right-hand side:

$$ a^*(n^2 - 2n + 1) + (7a^* - 2)(n - 1) + (6a^* - 2) + \frac{2a^*(n^2-6n+9) + 2(7a^*-2)(n-3) + 2(6a^*-2)}{n-2} + 2.$$

Comparing coefficients with the left-hand side $a^*n^2 + (7a^* - 2)n + (6a^* - 2),$ the values match for the $n^2, n,$ constant and $\frac{1}{n-2}$ terms as required.

\end{proof}

\begin{proof} [Proof of Proposition 3.3.] Recall that $a_k$ is defined such that $$s_k = a_kk^2 + (7a_k - 2)k + (6a_k - 2)  = a_k(k^2 + 7k + 6) - 2(k+1).$$ Then, let $b_k$ be the maximum absolute difference between any two of $a_k, a_{k+1}$ and $a_{k+2}.$ We will find upper bounds for the absolute differences $|a_k - a_{k-1}|, |a_{k+1} - a_k|$ and $|a_{k+1} - a_{k-1}|$ in terms of $b_{k-3}.$ By definition, the largest of these is an upper bound for $b_{k-1}.$ 

By re-arranging the definition, $a_k = \frac{s_k + 2(k+1)}{k^2 + 7k + 6}, a_{k-1} = \frac{s_{k-1} + 2k}{k^2 + 5k}$ and $a_{k-3} = \frac{s_{k-3} + 2(k-2)}{k^2 + k - 6}.$ Then

$$a_{k-1} - a_{k-3} = \frac{s_{k-1} + 2k}{k^2 + 5k} -  \frac{s_{k-3} + 2(k-2)}{k^2 + k - 6}.$$

The modulus of this gives a lower bound for the value of $b_{k-3},$ which may be equal to the absolute difference $|a_{k-1} - a_{k-3}|$ or instead to one of $|a_{k-1} - a_{k-2}|, |a_{k-2} - a_{k-3}|,$ should either of those be larger.

To write the difference $a_k - a_{k-1}$ in terms of $a_{k-1} - a_{k-3},$ we start with the fact that $s_k = 2 + s_{k-1} + \frac{2}{k-2}s_{k-3}, $ from \eqref{recurrence}. Substituting the expressions for the $s_k, s_{k-1}, s_{k-3}$ gives

$$a_k(k^2 + 7k + 6) - 2(k+1) = 2 + a_{k-1}(k^2 + 5k) - 2k + \frac{2}{k-2}\left( a_{k-3}(k^2 + k - 6) - 2(k-2) \right)$$

and, re-arranging,

$$a_k = \frac{ a_{k-1}k(k+5) + 2a_{k-3}(k+3)  }{ (k+6)(k+1) }.$$

This expression tells us that $a_k$ is a weighted average of the values of $a_{k-1}$ and $a_{k-3},$ and we therefore know that all values of $a_i$ for $i > k+2$ lie between the minimum and maximum values of $a_k, a_{k+1}, a_{k+2},$ which means that $(b_k)$ is a non-strictly decreasing sequence.

From here, $a_k - a_{k-1} = \frac{ a_{k-1}\left[ k(k+5) - (k+6)(k+1) \right] + 2a_{k-3}(k+3)  }{ (k+6)(k+1) },$ and we can write this as $a_k - a_{k-1} = \frac{(-2k-6)(a_{k-1} - a_{k-3})}{(k+6)(k+1)}$ as required. We now have the upper bound

$$|a_k - a_{k-1}| \leq \frac{2k+6}{(k+6)(k+1)}b_{k-3}.$$

Replacing $k$ with $k+1$ in the final equation gives $a_{k+1} - a_k = \frac{(-2k - 8)(a_k - a_{k-2})}{(k+7)(k+2)},$ and so $|a_{k+1} - a_{k}| \leq \frac{2k+8}{(k+7)(k+2)} b_{k-2}.$ Since $(b_k)$ is a non-increasing sequence, $b_{k-3}$ is at least as great as $b_{k-2},$ and so

$$|a_{k+1} - a_{k}| \leq \frac{2k+8}{(k+7)(k+2)} b_{k-3}.$$

Adding together the previous two equations and using the triangle inequality gives the following bound for $|a_{k+1} - a_{k-1}|;$

$$|a_{k+1} - a_{k-1}| \leq \left( \frac{2k+6}{(k+6)(k+1)} + \frac{2k+8}{(k+7)(k+2)}  \right) b_{k-3}.$$

The largest of these three differences is the value of $b_{k-1}.$ Since all the terms are positive, the bound for $|a_{k+1} - a_{k-1}|$ is the largest of the three, and so it is a bound for all of them:$$b_{k-1} \leq \left( \frac{2k+6}{(k+6)(k+1)} + \frac{2k+8}{(k+7)(k+2)}  \right) b_{k-3}.$$

This inequality can be re-arranged to $b_{k-1} \leq \frac{2}{5}\left( \frac{2}{k+2} + \frac{3}{k+6} + \frac{3}{k+7} + \frac{2}{k+1} \right)b_{k-3}.$ Then, $\frac{2}{5}\left( \frac{2}{k+2} + \frac{3}{k+6} + \frac{3}{k+7} + \frac{2}{k+1} \right)b_{k-3} \leq \frac{2}{5} \left( \frac{10}{k+1} \right) b_{k-3} = \frac{4}{k+1} b_{k-3},$ and hence $$b_{k-1} \leq \frac{4}{k+1}b_{k-3}.$$

This shows that the sequence $(b_k)$ has limit zero, and this means that there is a limit for the sequence $(a_k),$ which we name $a.$

We will obtain more specific forms for even and odd values of $k.$ Since $s_0 = 0, s_1 = 0$ and $s_2 = 1,$ the first values of $a_k$ are $a_0 = \frac{1}{3}, a_1 = \frac{2}{7}$ and $a_2 = \frac{7}{24}.$ These values give $b_0 = \frac{1}{21}$ (the largest difference of any two). The relation between $b_{k-1}$ and $b_{k-3}$ gives, when used iteratively, that, for $k = 2l$ with $l \in \mathbb{N};$

$$b_{2l} \leq \frac{2^l}{21(l+1)!}.$$

For $k = 2m + 1$ with $m \in \mathbb{N},$ the same method (using that $s_3 = 3, a_3 = \frac{11}{36}$ and so $b_1 = \frac{5}{252}$ gives $b_{2m+1} \leq \frac{15\sqrt{\pi} \cdot 2^{m-2}}{252 \cdot \Gamma (m + \frac{5}{2})},$ which re-arranges to

$$b_{2m+1} \leq \frac{ \frac{15}{252} 4^m  }{(2m+3)!!},$$

where $x!!$ denotes the double factorial function of $x$ (the product of $x$ and only every second positive integer less than $x$).

\end{proof}

\begin{proof} [Proof of Proposition 3.4.] We will prove by induction that the sequences $(p_n)$ and $(q_n)$ are increasing for $n \geq 9.$ First, the base steps will be that these sequences are increasing from $n=9$ to $n=23.$ The relevant calculations were done by computer and are given in the appendix.

For the inductive step, note that $p_k = \frac{s_k - s_{k-1}}{k}$ by the definition of $s_k.$ We define $A_k = a(k^2 + 7k + 6) - 2(k+1);$ $A_k$ is defined similarly to $s_k,$ except for the presence of $a$ instead of $a_k.$ The differences $|s_k - A_k|$ are small as the differences between the estimates $a_k$ and the true value of $a$ are small. We will use this to approximate $\frac{s_k - s_{k+1}}{k}.$

The absolute difference between $s_k$ and $A_k$ is given by $$\left| a_k(k^2 + 7k + 6) - 2(k+1) - \left( a(k^2 + 7k + 6) - 2(k+1)  \right) \right| = |a - a_k|(k^2 + 7k + 6) $$  $$\leq b_{k-1} (k^2 + 7k + 6).$$ The absolute differences $|s_{k-1} - A_{k-1}| \leq b_{k-2}(k^2+5k)$ and $|s_{k-2} - A_{k-2}| \leq b_{k-3}(k+4)(k-1)$ follow.

We therefore know that $s_k$ has a value within $[A_k - b_{k-1}(k^2 + 7k + 6), A_k + b_{k-1}(k^2 + 7k + 6)],$ and $s_{k-1}$ has a value within $[A_{k-1} - b_{k-2}(k^2 + 5k), A_{k-1} + b_{k-2}(k^2 + 5k)].$ We can then find a range for $s_k - s_{k-1};$ it is maximised when $s_k$ is as large as possible and $s_{k-1}$ is as small as possible, and vice versa for when it is minimised.

Since $A_k - A_{k-1} = a(k^2 + 7k + 6) - 2(k+1) - \left( a(k^2 + 5k) - 2k  \right) = 2ak + 6a - 2,$ and so $ \frac{A_k - A_{k-1}} {k} = 2a + \frac{6a-2}{k},$ the value of $\frac{A_k - A_{k-1}}{k}$ is increasing in $k,$ since $6a - 2$ is negative. It only remains to show that this increase is enough to outweigh the possible error of $b_{k-1}(k^2 + 7k + 6) + b_{k-2}(k^2+5k)$ from the approximation of $s_k$ by $A_k.$

We compare the increase of $A_k - A_{k-1}$ to the total error. As $j$ increases from $k$ to $k+1,$ $A_j - A_{j-1}$ increases from $2a + \frac{6a - 2}{k}$ to $2a + \frac{6a - 2}{k+1}.$ This is an increase by $\frac{6a-2}{k+1} - \frac{6a - 2}{k} = \frac{2-6a}{k^2+k}.$ Comparing with the error of $b_{k-1}(k^2 + 7k + 6) + b_{k-2}(k^2+5k),$ using the previously obtained inequalities, the increase is larger for $k = 22$ and $k = 23$ (by calculation). We will show by induction that the statement $P(k): \frac{2 - 6a}{k^2 + k} \geq b_{k-1}(k^2+7k+6) + b_{k-2}(k^2+5k)$ is true for all $k \geq 22,$ by showing that $P(k)$ implies $P(k+2).$

The statement $P(k)$ can be re-arranged to $$2 - 6a \geq b_{k-1}(k^2+7k+6)(k^2+k) + b_{k-2}(k^2 + 5k)(k^2 + k).$$

Because $\frac{4(k^2 + 9k + 14)(k^2 + 5k + 6)}{(k+3)(k^2 + 7k + 6)(k^2 + k)} < 1$ as long as $k \geq 6,$ and $\frac{4(k^2+7k+6)(k^2+5k+6)}{(k+2)(k^2+5k)(k^2+k)} < 1$ as long as $k \geq 7,$ we can deduce from $P(k)$ that
\begin{align}
2 - 6a &\geq \frac{4(k^2 + 9k + 14)(k^2 + 5k + 6)}{(k+3)(k^2 + 7k + 6)(k^2 + k)} b_{k-1}(k^2+7k+6)(k^2+k) \notag \\ &+ \frac{4(k^2+7k+6)(k^2+5k+6)}{(k+2)(k^2+5k)(k^2+k)} b_{k-2}(k^2 + 5k)(k^2 + k). \notag
\end{align}

Proposition 3.3 tells us that $b_{k+1} \leq \frac{4}{k+3}b_{k-1}$ and $b_k \leq \frac{4}{k+2}b_{k-2}.$ Using these results and cancelling gives us that

$$2 - 6a \geq b_{k+1}(k^2 + 9k + 14)(k^2 + 5k + 6) + b_k(k^2 + 7k + 6)(k^2 + 5k + 6),$$

and, dividing by $k^2 + 5k + 6 = (k+2)^2 + (k+2),$

$$\frac{2 - 6a}{(k+2)^2 + (k+2)} \geq b_{k+1}((k+2)^2+7(k+2)+6) + b_{k}((k+2)^2+5(k+2)),$$

as required. Therefore, $P(k)$ implies $P(k+2),$ and since $P(22)$ and $P(23)$ are true, $P(k)$ is true for all $k \geq 22.$

We will use a similar argument to show that the sequence $(q_n)$ is increasing for $n \geq 23.$ Using the relation $q_k = \frac{1}{k} \left( 1 + \frac{2}{k}s_{k-1} \right),$ and that $|s_{k-1} - A_{k-1}| \leq b_{k-2}(k^2 + 5k),$ we know that $q_k \in [\frac{1}{k} \left( 1 + \frac{2}{k}A_{k-1} \right) - \frac{2}{k^2} b_{k-2}(k^2+5k), \frac{1}{k} \left( 1 + \frac{2}{k}A_{k-1} \right) + \frac{2}{k^2} b_{k-2}(k^2+5k)],$ or equivalently, using the definition $A_{k-1} = a(k^2 + 5k) - 2k,$

$$q_k \in  [2a + \frac{10a-3}{k} - \frac{2}{k^2} b_{k-2}(k^2+5k) ,   2a + \frac{10a-3}{k} + \frac{2}{k^2} b_{k-2}(k^2+5k) ].$$

As $j$ increases from $k$ to $k+1,$ our approximate value for $q_j$ increases from $2a + \frac{10a - 3}{k}$ to $2a + \frac{10a - 3}{k+1},$ an increase of $\frac{10a - 3}{k+1} - \frac{10a - 3}{k} = \frac{3 - 10a}{k(k+1)}.$ Computationally, comparing this with the possible error shows that, for $k = 22$ and $k=23,$ the increase is greater than the error. We will now proceed similarly to before, and demonstrate that, for the statement $P(k): \frac{3-10a}{k(k+1)} \geq \frac{2}{k^2}b_{k-2}(k^2 + 5k),$ we have that $P(k)$ implies $P(k+2).$

Multiplying both sides by $k(k+1)$ gives

$$3-10a \geq \frac{2b_{k-2}(k^2+5k)(k^2+k)}{k^2}.$$

Since $\frac{4(k^2+5k+6)(k^2+9k+14)k^2}{(k+2)(k^2+5k)(k^2+k)(k^2+4k+4)} < 1$ as long as $k \geq 5,$ we can write

$$3-10a \geq \frac{4(k^2+5k+6)(k^2+9k+14)k^2}{(k+2)(k^2+5k)(k^2+k)(k^2+4k+4)}  \frac{2b_{k-2}(k^2+5k)(k^2+k)}{k^2}.$$

Using that $b_k \leq \frac{4}{k+2}b_{k-2}$ (from Proposition 3.3) and cancelling, we have

$$3 - 10a \geq \frac{2b_k (k^2 + 9k + 14)(k^2 + 5k + 6)}{k^2 + 4k + 4},$$

and dividing both sides by $k^2 + 5k + 6) = (k+2)(k+3),$

$$\frac{3-10a}{(k+2)(k+3)} \geq \frac{2}{(k+2)^2}b_k((k+2)^2 + 5(k+2)),$$

which is the statement $P(k+2),$ therefore $P(k)$ implies $P(k+2),$ as required. Since $P(22)$ and $P(23)$ are true, $P(k)$ is true for all $k \geq 22$ by induction.

\end{proof}

\begin{proof} [Proof of Proposition 3.5.] We prove by strong induction the induction hypothesis $P(k):$ that, when $X$ starts in $MS(P_k),$ guesses of $2$ and $k-1$ are optimal for them (i.e. there are no other guesses that have a higher winrate).

We have thirteen base steps, covering the cases where $9 \leq n \leq 23,$ detailed in the appendix.

For the inductive step, we start by writing $P(\text{X wins }MS(P_{k+1})\text{ choosing }x).$ Since there is a probability $\frac{x-1}{k+1}$ that the mine is less than $x,$ in which case the game reduces to size $x-1,$ and $\frac{k+1-x}{k+1}$ that it is greater than $x,$ in which case the game reduces to size $k+1-x,$ we have   $$P(\text{X wins }MS(P_{k+1})\text{ choosing }x) = \frac{x-1}{k+1} q_{x-1} + \frac{k+1-x}{k+1} q_{k+1-x}.$$

We therefore need to find the integers $x$ between $1$ and $k+1$ that maximise $F(x) = (x-1)q_{x-1} + (k+1-x)q_{k+1-x}.$ By Proposition 3.4, and the casework in the appendix (specifically, note that $q_7$ lies between $q_{22}$ and $q_{23}$), we can find the rank order of the $q_n~ 's:$

\begin{enumerate}
\item The largest four values are $q_1 > q_4 > q_5 > q_8;$
\item For any $i,j \notin \{ 1,4,5,7,8 \}, $ with $i > j,$ we have $q_i > q_j;$
\item $q_{23} > q_7 > q_{22}.$
\end{enumerate}

If $k > 23,$ then the best possible sizes for regions to divide the game into are, with the best sizes first, sizes $1, 4, 5, 8, k, k-1$ and then other choices. If we choose $x=2,$ then the possible divisions are of sizes $1$ and $k-1.$ Therefore, if there is any better possibility, at least one of the possible resulting sizes of the game must be a size higher in the list than $k-1$ (otherwise all cases are worse than the worst case scenario when choosing $x=2,$ assuming the mine itself is not chosen, which always has equal probability.) The candidate choices for $x$ that fit this are $x=1$ (possible division into size $k$), $x=5$ (possible divisions into sizes $4$ and $k-4$), $x=6$ (possible divisions into sizes $5$ and $k-5$) and $x=9$ (possible divisions into sizes $8$ and $k-8$).
We will focus only on choices of $x$ less than  $\frac{k+1}{2};$ due to the symmetry in the game, choices of $x$ higher than this will have equivalent winrates to those already tested.

The relevant values of $F(x)$ are: $F(1) = kq_k, F(2) = q_1 + (k-1)q_{k-1}, F(5) = 4q_4 + (k-4)q_{k-4}, F(6) = 5q_5 + (k-5)q_{k-5}$ and $F(9) = 8q_8 + (k-8)q_{k-8}.$ Substituting the values for $q_1$ to $q_8$ gives $F(2) = 1 + (k-1)q_{k-1}, F(5) = 2.5 + (k-4)q_{k-4}, F(6) = \frac{80}{27} + (k-5)q_{k-5}$ and $F(9) = \frac{115}{24} + (k-8)q_{k-8}.$ Since $q_n \geq \frac{1}{2}$ for any $n,$ and $q_{k-1} > q_{k-4} > q_{k-5} > q_{k-8}$ from Proposition 3.3, $F(2)$ is greater than $F(5), F(6)$ and $F(9).$ This demonstrates that the choice $x=2$ is better than any other $x$ less than $\frac{k+1}{2}.$ We now just need to compare $F(2)$ with $F(1);$ we require that $F(2) - F(1) \geq 0.$

First, write $F(2) - F(1) = 1 + (k-1)q_{k-1} - kq_k.$ We will substitute $q_k$ and $q_{k-1}$ with expressions involving $s_k~'s.$ Recall that $q_n = \frac{1}{n}\left( 1 + \frac{2}{n} \sum_{i=1}^{n-1} ip_i \right) = \frac{1}{n}\left( 1 + \frac{2}{n} s_{n-1} \right).$ Substituting this in gives us that $F(2) - F(1) = 1 + (k-1)\frac{1}{k-1}\left( 1 + \frac{2}{k-1}s_{k-2} \right) - k \frac{1}{k} \left(1 + \frac{2}{k} s_{k-1}\right),$ which simplifies to $$F(2) - F(1) = 1 - 2 \left( \frac{s_{k-1}}{k} - \frac{s_{k-2}}{k-1} \right).$$

This is approximately equal to $1 - 2 \left( \frac{A_{k-1}}{k} - \frac{A_{k-2}}{k-1} \right) = 1 - 2(a(k+5) - 2 - a(k+4) - 2) = 1 - 2(a-4),$ with a possible error of $b_{k-2} (k+5) + b_{k-3}(k+4).$ The value of $1 - 2(a-4)$ is positive (and it is approximately $8.4$) and the decay of the sequence $(b_k)$ is easily $o(1/k)$ from Proposition 3.3, so the positive difference firmly outweighs the potential error for $k > 23.$ Specifically, for $k \geq 3,$ let $P(k)$ be the statement that $1 - 2(a-4) > b_{k-2}(k+5) + b_{k-3}(k+4),$ and we will prove that $P(k)$ implies $P(k+2).$ From Proposition 3.3, we have $b_{k-1} \leq \frac{4}{k+1}b_{k-3}$ and $b_k \leq \frac{4}{k+2} b_{k-2}.$ Then, as long as $k \geq 4,$ $\frac{4(k+7)}{(k+2)(k+5)} \leq 1$ and $\frac{4(k+6)}{(k+1)(k+4)} \leq 1.$ By the inductive hypothesis $P(k),$ we can therefore write

$$1 - 2(a-4) > \frac{4(k+7)}{(k+2)(k+5)}(k+5)b_{k-2} + \frac{4(k+6)}{(k+1)(k+4)}(k+4)b_{k-3}.$$

Simplifying this and subtituting the results from Proposition 3.3 gives $$1 - 2(a-4) > \frac{4}{k+2}b_{k-2}(k+7) + \frac{4}{k+1}b_{k-3}(k+6) > b_k(k+7) + b_{k-1}(k+6),$$

and the statement $P(k+2)$ is true.

The statements $P(3)$ and $P(4)$ are true; the low values of $a_k$ are $a_0 = \frac{1}{3}, a_1 = \frac{2}{7}, a_2 = \frac{7}{24}, a_3 = \frac{11}{36}, a_4 = \frac{3}{10}.$ Then, the low values of $b_k$ are $b_0 = \frac{1}{21}, b_1 = \frac{5}{252}, b_2 = \frac{1}{120}.$ Then we can verify $P(3): 1 - 2(a-4) > 8b_1 + 7b_0 = \frac{31}{63}$ and $P(4): 1 - 2(a-4) > 9b_2 + 8b_1 = \frac{589}{2520},$ are true. So, $P(k)$ is true for all $k \geq 3$ by induction.

This demonstrates that the choice $x=2$ is better than $x=1.$

\end{proof}

\begin{proof} [Proof of Proposition 3.6.] Since $p_k = \frac{s_k - s_{k-1}}{k},$ and the difference between this and $\frac{A_k - A_{k-1}}{k} = 2a + \frac{6a-2}{k}$ approaches zero and $\frac{6a-2}{k}$ approaches zero as $k$ approaches infinity, the limit for the sequence $(p_k)$ is $2a.$ By the same reasoning, using that the limit of $\frac{10a-3}{k}$ is zero, the limit for the sequence $(q_k)$ is also $2a.$

\end{proof}

\subsection{Nature of the other solutions of the $s_n$ recurrence}

We described in Proposition 3.2 one of the three linearly independent solutions of \eqref{recurrence}. We will demonstrate that the other two solutions are not hypergeometric. A hypergeometric solution $x(n)$ to a recurrence is one where there is a rational function $r(n)$ (an algebraic fraction with polynomial numerator and denominator) such that $x(n+1) = r(n)x(n)$ for all large enough $n \in \mathbb{N}.$

\begin{proposition} Other than the family of solutions described in Proposition 3.2, there are no other hypergeometric solutions of \eqref{recurrence} for $s_n.$
\end{proposition}

\begin{proof}[Proof of Proposition 3.7.] We will use the algorithm `Hyper' of Marko Petkovsek \cite{hypergeometric} to find all the hypergeometric solutions of \eqref{recurrence}. Firstly, we find an auxiliary, homogenised recurrence. By substituting higher values of $n,$ we have $s_{n+3} = s_{n+2} + \frac{2}{n+1}s_n + 2$ and $s_{n+4} = s_{n+3} + \frac{2}{n+2}s_{n+1} + 2,$ and subtracting the second of these from the first gives $s_{n+4} - s_{n+3} = s_{n+3} - s_{n+2} + \frac{2}{n+2}s_{n+1} - \frac{2}{n+1}s_n,$ which re-arranges to the auxiliary recurrence

$$s_{n+4} - 2s_{n+3} + s_{n+2} - \frac{2}{n+2}s_{n+1} + \frac{2}{n+1}s_n = 0.$$

This homogenisation increases the order of \eqref{recurrence}, and will produce an additional spurious linearly independent solution - this solution is $\alpha(k+1)$ for constant $\alpha,$ which is not a solution of the original recurrence. We will use the Hyper algorithm to identify the two linearly independent polynomial solutions and establish that the remaining two solutions are not hypergeometric.

Using the definitions in \cite{hypergeometric}, we have $p_0(n) = 2(n+2), p_1(n) = -2(n+1), $ \linebreak $p_2(n) = (n+1)(n+2), p_3(n) = -2(n+1)(n+2)$ and $p_4(n) = (n+1)(n-2).$ Then, the monic factors $A(n)$ of $p_0(n)$ are $n+2$ and $1,$ and the monic factors $B(n)$ of $p_4(n-3)$ are $n-1, n-2$ and $1.$ Using the definition $P_i(n) = p_i(n) \prod_{j=0}^{i-1} A(n+j) \prod_{j=i}^{d-1}B(n+j),$ we have

\begin{align}
P_0(n) &= 2(n+2)B(n)B(n+1)B(n+2)B(n+3), \notag \\ \notag  P_1(n) &= -2(n+1)A(n)B(n+1)B(n+2)B(n+3), \\ \notag P_2(n) &= (n+1)(n+2)A(n)A(n+1)B(n+2)B(n+3),  \\ \notag P_3(n) &= -2(n+1)(n+2)A(n)A(n+1)A(n+2)B(n+3),  \\ \notag P_4(n) &= (n+1)(n+2)A(n)A(n+1)A(n+2)A(n+3). \notag
\end{align}

There are now six cases to study: $A(n)$ can be chosen to be either $n+2$ or $1,$ and $B(n)$ can be chosen to be any of $n-1,$ $n-2$ or $1.$ An overview of the results is as follows: the cases where exactly one of $A(n), B(n)$ is $1$ are degenerate (and yield no solutions); the cases with both $A(n) = n+2$ and $B(n) \neq 1$ also yield no solutions, with lengthier work; finally, the case with $A_n = B_n = 1$ yields the polynomial solution of the original recurrence and the linearly independent spurious polynomial solution. The auxiliary recurrence has order four, and must have four linearly independent solutions; since the others are not found by this algorithm, they are not hypergeometric. We will now detail these calculations.

In the case $A_n = B_n = 1,$ the expressions for the $P_i(n)$ are equal to those of the $p_i(n),$ and then $m=2$ (the highest degree of any of the $P_i(n)$ is $2.$ Then the coefficients of $n^2,$ as defined in \cite{hypergeometric}, are $\alpha_0 = \alpha_1 = 0, \alpha_2 = 1, \alpha_3 = -2$ and $\alpha_4 = 1;$ the solutions of $\sum_{i=0}^4 \alpha_i Z^i = 0,$ i.e., $Z^2 - 2Z^3 + Z^4 = 0,$ are $Z = 0, 1.$ Then, the next step of the algorithm, using $Z=1,$ recovers the original recurrence problem; we must solve the recurrence $\sum_{i=0}^4 p_i(n)C(n+i)$ (which is identical to the auxiliary recurrence) for polynomial solutions $C;$ solving by comparing coefficients gives the linearly independent solutions,  for any constants $a$ and $b,$ of $a(n^2 + 7n + 6)$ (which is the solution exhibited in Proposition 3.2) and $b(n+1),$ (which is the spurious solution, and is not a solution of the original recurrence).

An example of a degenerate case where one of $A(n), B(n)$ is $1$ is the case with $A(n) = 1, B(n) = n-1.$ In this case we have

\begin{align}
P_0(n) &= 2(n-1)n(n+1)(n+2)^2\notag , \\  P_1(n) &= -2n(n+1)^2(n+2)\notag , \\ P_2(n) &= (n+1)^2(n+2)^2\notag , \\ P_3(n) &= -2(n+1)(n+2)^2\notag , \\ P_4(n) &= (n+1)(n+2)\notag , \\
\end{align}

yielding $m=5$ and $\alpha_0 = 1,$ with $\alpha_i = 0$ for $i \in \{ 1,2,3,4 \}.$ Then, there are no solutions of $\sum_{i=0}^4 \alpha_i Z^i = 0,$ i.e. $Z^0 = 0.$ The calculations are virtually identical for the case with $A(n) = 1, B(n) = n-2;$ the degrees of all of the $P_i(n),$ and all the values of the $\alpha_i,$ remain the same.

The case with $A(n) = n+2, B(n) = 1$ is also very similar; for that case, we have
\begin{align}
P_0(n) &= 2(n+2) \notag , \\  P_1(n) &= -2(n+1)(n+2) \notag , \\ P_2(n) &= (n+1)(n+2)^2(n+3) \notag , \\ P_3(n) &= -2(n+1)(n+2)^2(n+3)(n+4) \notag , \\ P_4(n) &= (n+1)(n+2)^2(n+3)(n+4)(n+5) \notag , \\
\end{align}

yielding $m=6$ and $\alpha_4 = 1,$ with $\alpha_i = 0$ for $i \in \{ 0,1,2,3 \}.$ Then, there are no non-zero solutions of $\sum_{i=0}^4 \alpha_i Z^i = 0,$ i.e. $Z^4 = 0.$

This leaves the most involved cases where neither of $A(n)$ nor $B(n)$ are $1.$ In the case with $A(n) = n+2, B(n) = n-1,$ we have
\begin{align}
P_0(n) &= 2(n-1)n(n+1)(n+2)^2 \notag , \\  P_1(n) &= -2n(n+1)^2(n+2)^2 \notag , \\ P_2(n) &= (n+1)^2(n+2)^3(n+3) \notag , \\ P_3(n) &= -2(n+1)(n+2)^3(n+3)(n+4) \notag , \\ P_4(n) &= (n+1)(n+2)^2(n+3)(n+4)(n+5) \notag , \\
\end{align}

yielding $m=6,$ and, identically to in the case with $A(n) = B(n) = 1,$ $\alpha_0 = \alpha_1 = 0, \alpha_2 = 1, \alpha_3 = -2$ and $\alpha_4 = 1,$ and the non-zero solution $Z=1$ of $\sum_{i=0}^4 \alpha_i Z^i = 0.$

We factor out $(n+1)(n+2)^2$ and then search for polynomial solutions for $$(2n^2 - 2n)C(n) + (-2n^2-2n)C(n+1) + (n^3 + 6n^2 + 11n + 6)C(n+2) + (-2n^3 - 18n^2 - 52n - 48)C(n+3)$$$$ + (n^3 + 12n^2 + 47n + 60)C(n+4),$$

using Algorithm Poly from \cite{hypergeometric}. The values of the $c_{i,0}$ are the coefficients of the $n^3$ terms of the functions multiplied by $C(n+i);$ i.e. $c_{0,0} = 0, c_{1,0} = 0, c_{2,0} = 1, c_{3,0} = -2, c_{4,0} = 1,$ with the $c_{i,1}$ being similarly defined by the coefficients of the $n^2$ terms, and so on for the $c_{i,2}$ by the $n$ coefficients and $c_{i,3}$ by the constants.

For the $s=0$ step, we have $b_0^{(0)} = \sum_{i=0}^4 i^0 c_{i,0} = 0.$ Then, for the $s=1$ step, we have $b_0^{(1)} = \sum_{i=0}^4 i^0 c_{i,1} = 0$ (note that for the Poly algorithm, the value of $0^0$ is taken to be $1,$ and $b_1^{(1)} = \sum_{i=0}^4 i c_{i,0} = 0.$ Then, for the $s=2$ step, $b_0^{(2)} = \sum_{i=0}^d i^0 c_{i,2} = -2 -2 +11 - 52 + 47 = 2,$ $b_1^{(2)} = \sum_{i=0}^d i c_{i,1} = -2 + 2 \cdot 6 - 3 \cdot 18 + 4 \cdot 12 = 4,$ and $b_2^{(2)} = \sum_{i=0}^d i^2 c_{i,0} = 1 \cdot 2^2 - 2 \cdot 3^2 + 1 \cdot 4^2 = 2.$

Since the values were not zero, we move on at this step, with $b_0^{(2)} = 2, b_1^{(2)} = 4, b_2^{(2)} = 2.$ We must solve the polynomial $\sum_{j=0}^2 {N \choose j} b_j^{(2)} = 0$ for non-negative integer roots $N.$ This polynomial is $N^2 + 3N + 2 = 0,$ which has no such roots, so the output of the algorithm is that there are no hypergeometric solutions derived from this case.

Finally, we deal with the case $A(n) = n+2, B(n) = n-2,$ very similarly. We have

\begin{align}
P_0(n) &= 2(n-2)(n-1)n(n+1)(n+2) \notag , \\
P_1(n) &= -2(n-1)n(n+1)^2(n+2) \notag , \\
P_2(n) &= n(n+1)^2(n+2)^2(n+3) \notag , \\
P_3(n) &= -2(n+1)^2(n+2)^2(n+3)(n+4) \notag , \\
P_4(n) &= (n+1)(n+2)^2(n+3)(n+4)(n+5) \notag , \\
\end{align}

yielding $m=6$ and $\alpha_0 = \alpha_1 = 0, \alpha_2 = 1, \alpha_3 = -2$ and $\alpha_4 = 1$ with $Z=1$ once again. We factor out $(n+1)(n+2)$ and solve

$$(2n^3 - 6n^2 + 4n)C(n) + (-2n^3 + 2n)C(n+1) + (n^4 + 6n^3 + 11n^2 + 6n)C(n+2) + $$
$$(-2n^4 - 20n^3 - 70n^2 - 100n - 48)C(n+3) + (n^4 + 14n^3 + 71n^2 + 154n + 120)C(n+4).$$

The values of $c_{i,j}$ will be defined as the coefficient of $n^{4-j}$ in the function multipled by $C(n+i).$ Then, for the $s=0$ step,  $b_0^{(0)} = \sum_{i=0}^4 i^0 c_{i,0} = 0.$ For the $s=1$ step, $b_0^{(1)} = \sum_{i=0}^4 i^0 c_{i,1} = 0$ and $b_1^{(1)} = \sum_{i=0}^4 i c_{i,0} = 0.$ For the $s=2$ step, $b_0^{(2)} = \sum_{i=0}^d i^0 c_{i,2} = -6 + 11 - 70 + 71 = 6, b_1^{(2)} = \sum_{i=0}^d i c_{i,1} = -2 + 2 \cdot 6 - 3 \cdot 20 + 4 \cdot 14 = 6,$ and $b_2^{(2)} = \sum_{i=0}^d i^2 c_{i,0} = 1 \cdot 2^2 - 2 \cdot 3^2 + 1 \cdot 4^2 = 2.$ The polynomial $\sum_{j=0}^2 {N \choose j} b_j^{(2)} = 0$ is $N^2 + 5N + 6 = 0,$ which has no non-negative integer roots, so again there are no hypergeometric solutions derived from this case.

We have concluded that the only hypergeometric solutions of the auxiliary recurrence are the two polynomial solutions, one of which is one of the three solutions of the original third order recurrence. We therefore know that the other two solutions are non-hypergeometric (we also saw in the proof of Proposition 3.3 that they are rapidly decaying.)

\end{proof}

\section{Stars}

We will now describe the solution for the case where $G$ is a star, both with two perfect players and with random and exploitative players.

Theorem 3.1 of \cite{treesearch} implies that, for the Boppana-Lewis tree search game with two perfect players, the winning probability for each player depends only on the number of vertices $n,$ and not the structure of the tree. In our mis\`{e}re searching game this is false, and we will see a counterexample; the stars and paths will not have the same winning probabilities.

For our study of an exploitative player $X$ versus a random player $R,$ write $P(T)$ for the probability that player $X$ wins $MS(T)$ as the first player, and $Q(T)$ for the probability that they win as the second player.

A star graph $S_n$ has one `root' vertex, and its remaining $n-1$ vertices are leaves connected to this root. Theorem 3.1 of \cite{treesearch} implies that, in the Boppana-Lewis search, the first player wins with probability $\frac{1}{2}$ for even $n$ and $\frac{(n+1)/2}{n}$ for odd $n;$ their strategy will always be to guess a leaf of the star, as guessing the root would cut down the entire tree and reveal the target vertex to the opposing player.

In our mis\`{e}re searching game, the optimal players should guess the root vertex on their turn, which will leave only the poisoned vertex for the opponent. Therefore, the first player trivially wins with probability $\frac{n-1}{n},$ losing only if the root vertex is poisoned. This demonstrates that an equivalent result to Theorem 3.1 of \cite{treesearch} is not true for the mis\`{e}re search, as the probability that the first player wins $MS(S_n)$ is $\frac{n-1}{n}$, but for $MS(P_n)$ it is not.

If there is a random player $R$ and exploitative player $X,$ then $X$ similarly wins $MS(S_n)$ as the first player with probability $\frac{n-1}{n}$ using the same strategy; $P(S_n) = \frac{n-1}{n}.$ On the other hand, if the random player goes first, they will more than likely randomly select a leaf as their guess. The possibilities, assuming that $n \geq 4$ (otherwise, the graph is simply a path graph) are:
\begin{enumerate}
\item Player $R$ guesses a leaf, which has probability $\frac{n-1}{n}.$ Then, they win only with probability $\frac{1}{n}$ (if the poisoned vertex is the root vertex, which will be guessed by $X$ on their next turn).
\item Player $R$ guesses the root, which has probability $\frac{1}{n}.$ Then, they win with probability $\frac{n-1}{n}.$
\end{enumerate}
In total, the probability that $R$ wins $MS(S_n)$ as the first player is $\frac{2(n-1)}{n^2},$ and so $Q(S_n) =\frac{n^2 - 2n + 2}{n^2}.$ We have $\lim_{n \to \infty} P(S_n) = \lim_{n \to \infty} Q(S_n) = 1;$ as the size of the star approaches infinity, the probability that the exploitative player wins, going first or second, approaches $1.$ In the Boppana-Lewis game, the equivalent result was that the limit of the exploitative player's winning probability was $\frac{2}{3};$ cutting the root is a poor choice in that game, and the exploitative player still has an advantage as the random player may choose to cut the root, but the advantage is not as decisive as in the mis\`{e}re game where the exploitative player will almost surely win the game after their first turn.

We will now prove that, for any number of vertices $n,$ there is no better possible tree structure for the exploitative player than a star, regardless of whether they go first or second.

\begin{proof}[Proof of Theorem 1.3.] Firstly, clearly $P(T) \leq \frac{n-1}{n}$ for any tree $T$ on $n$ vertices, as $X$ always loses on their first turn with probability $\frac{1}{n}$ regardless of the structure of the tree. Since $P(S_n) = \frac{n-1}{n},$ this demonstrates that $P(T) \leq P(S_n).$

Let $n_G$ be the number of vertices in a general tree $G$ (and therefore we write $n_T = n.)$ To prove that $Q(T) \leq Q(S_n),$ we start by letting the first vertex chosen by $R$ be $v,$ and writing an expression for $Q(T)$ in terms of $P(S)$ for the different possible tree components $S$ that can emerge from $R's$ choice of $v.$ For each possible choice of $v,$ which occurs with probability $\frac{1}{n_T},$ we sum over these components $S,$ considering the probability that this component remains, which is $\frac{n_S}{n_T}$ (as the component remains if and only if one of its $n_S$ vertices is poisoned) multiplied by the probability that the exploitative player wins from there when the game continues on $S,$ which is $P(S).$ For each choice of $v,$ there is also a probability $\frac{1}{n_T}$ that $v$ is poisoned, in which case $X$ wins immediately, and we add this to our sum. This gives

$$Q(T) = \frac{1}{n_T} \sum_{v \in T} \left( \sum_{S \in C(T \setminus v)} \frac{n_S}{n_T} P(S) + \frac{1}{n_T} \right).$$

Using the previously proven $P(S) \leq \frac{n_S-1}{n_S}$ from the first part of the theorem, we have

$$Q(T) \leq \frac{1}{n_T} \sum_{v \in T} \left( \sum_{S \in C(T \setminus v)} \frac{n_S}{n_T} \frac{n_S-1}{n_S} + \frac{1}{n_T} \right).$$

Then

\begin{align}
Q(T) &\leq \frac{1}{n_T} \sum_{v \in T} \left( \sum_{S \in C(T \setminus v)} \frac{n_S-1}{n_T}  + \frac{1}{n_T} \right) \notag \\
&= \frac{1}{n_T^2} \sum_{v \in T}  \sum_{S \in C(T \setminus v)} (n_S - 1)  + \frac{1}{n_T}.\notag \\
\end{align}

We have $\sum_{S \in C(T \setminus v)} n_S = n_T - 1,$ and so $$\sum_{S \in C(T \setminus v)}(n_S - 1) = n_T - 1 - |C(T \setminus v)| = n_T - 1 - \text{deg}_T(v).$$

And so, since there are $n_T$ choices of $v,$ and the total degree of the vertices in a tree with $n$ vertices is $2(n-1):$
\begin{align}
Q(T) &\leq \frac{1}{n_T^2} \left( n_T(n_T - 1) \sum_{v \in T} \text{deg}_T(v) \right) + \frac{1}{n_T} \notag \\
&= \frac{1}{n_T^2} \left( n_T(n_T - 1) - \sum_{v \in T} \text{deg}_T(v) \right) + \frac{1}{n_T} \notag \\
&= \frac{1}{n_T^2}  (n_T - 2)(n_T - 1)  + \frac{n_T}{n_T^2} \notag \\
&= \frac{n_T^2 - 2n_T + 2}{n_T^2} = Q(S_n). \notag
\end{align}

\end{proof}

\section{Possible directions for future research}

\subsection{Extensions to spider graphs, caterpillar graphs and multi-pile Nim analogues}

The main question emerging from this work is to explain the behaviour of the mis\`{e}re search game on other trees. We will suggest two possible types of trees for study, with the idea that results on these trees may help to develop an understanding of how various types of connectivity present in trees can change the results, and how this behaviour can be described for a general tree. We will demonstrate how the game played on these types of trees can be coupled with a probabilistic game reminiscent of multi-pile Nim.

A spider (also known as a starlike graph), similarly to a star graph, has a single root vertex, and multiple `legs;' paths of vertices extending from the root.  Let $S_{\lambda_1, \lambda_2, ..., \lambda_k}$ denote the spider with $k$ legs of non-increasing lengths $\lambda_1, \lambda_2, ..., \lambda_k$ \cite{spider}. We will present a coupling between the game played on a spider and a multi-pile Nim style game.

\subsubsection{Nim-style coupling for the spider graph}

For a spider graph, write the current state of the game as a set of the leg lengths, i.e. $\{\lambda_1, \lambda_2, ..., \lambda_k\}.$ On a player's turn, they have two options:
\begin{itemize}
\item Choose one of the legs $\lambda_i$ for some $i,$ and a positive integer `cut' value $c \leq \lambda_i.$
\item Choose to cut the tree down (equivalent to guessing the root vertex).
\end{itemize}
In the first case, the player is guessing the $c$th vertex from the end of the relevant leg. With probability $\frac{n-c}{n},$ the poisoned vertex will be in the large remaining tree component, $c$ vertices will be cut from the end of the leg, and the value $\lambda_i$ in the list will be replaced by $\lambda_i - c.$ With probability $\frac{c-1}{n},$ the poisoned vertex occurs further towards the end of the path than the guessed vertex, and the game reduces to a path of length $c.$ Finally, as always, it is possible that the player guessed the poisoned vertex and loses (probability $\frac{1}{n}).$

In the second case where the player guesses the root vertex, they will lose with probability $\frac{1}{n},$ and otherwise the game will reduce to a path with a length of one of the $\lambda_i$ values, with probabilities $\frac{\lambda_i}{n}$ for each value.

This game, played on sets of numbers, is isomorphic to the mis\`{e}re search on a spider, and it may be that arguments involving concepts similar to Nim-sums, a bitwise-XOR calculation \cite{nim}, is useful, but the situation is complicated by the probabilistic nature of the game and the additional possible move of cutting down the entire tree.

A caterpillar graph features a path, but the vertices on the path may have any number of additional leaf vertices attached to them \cite{caterpillar}. A similar idea can be used to write a coupling between the tree searching games on caterpillar graphs and a multi-pile Nim style game, with extra possible moves that remove multiple piles at once with some probability.

\subsection{Other extensions}

Boppana and Lewis suggested several other possible extensions other than the mis\`{e}re search. Among these, for example, were a model where players guess edges rather than vertices, a model more closely reminiscent of the random process of `cutting down trees' of \cite{meir-moon}, a new setting for the `gold-grabber game' of Seacrest and Seacrest \cite{gold} and a continuous model where players must guess within a distance $\epsilon$ of a target in the Euclidean space $\mathbb{R}_n.$ They also pointed out that, while study so far has focused only on trees, it is also possible to play a similar game on general graphs, by using an extension of binary search from \cite{graph_binary_search}, wherein guessers are told which edge from their vertex is on the shortest path to the secret vertex. 

Guess What on Novel Games also allows for the possibility of more than two players \cite{novelgames}. In that case, when a player guesses the mine, they are eliminated and do not make any more guesses. The other players do not learn the eliminated player's guess, only that they were eliminated. This greatly increases the state space of the game, as the information that a player playing a known strategy has been eliminated will change the probabilities that each number, or vertex, is poisoned.

\section{Appendix}

\textbf{Base case calculations of $\bm{p_n}$ and $\bm{q_n}$ values, for low $\bm{n.}$}

First, $p_1 = 0$ ($X$ must guess the mine) and for the same reason $q_1 = 1.$ Therefore, $s_1 = 0.$ Then, $p_2 = q_2 \cdot \frac{1}{2},$ and $p_3 = \frac{2}{3}$ ($X$ guesses $2$ and wins if it is not the mine). For $q_3, R$ wins with probability $\frac{2}{3}$ if they choose $2,$ and otherwise with probability $\frac{1}{3},$ giving $q_3 = \frac{1}{3} \cdot \frac{1}{3} + \frac{2}{3} \cdot \frac{2}{3} = \frac{5}{9}.$ Using these $p_i$ values, we have $s_2 = 1$ and $s_3 = 3.$

From there, the relations $p_n = \frac{2}{n} + \frac{2}{n(n-2)}s_{n-3}$ and $q_n = \frac{1}{n}\left( 1 + \frac{2}{n}s_{n-1}\right)$ can be used. We used the following Python program to calculate the values in the sequences $(s_n), (p_n), (q_n)$:

\begin{lstlisting}[breaklines]
#### S values (recurrence)

maxN = 1000000 # Calculate up to this value of n.
print_precision = 100000 # Print the output when n & print_precision is 0.

print("S VALUES\n")

S = ["n = 0", 0, 1, 3]

for i in range(4,maxN + 1):
    S.append(2 + S[i-1] + 2*S[i-3]/(i-2))
    if i % print_precision == 0:
        print(i, S[i])

#### p values

print("\np_k VALUES\n")

p = ["n = 0",0]

for i in range(2,maxN + 1):
    p.append( (S[i] - S[i-1])/i )
    if i % print_precision == 0:
        print(i, p[i])

#### q values

print("\nq_k VALUES\n")

q = ["n = 0",1]

for i in range(2,maxN + 1):
    q.append( (1/i)*(1 + (2/i)*S[i-1]) )
    if i % print_precision == 0:
        print(i, q[i])

#### a_k values

print("\na_k VALUES\n")

a = ["n=0"]

for i in range(1,maxN + 1):
    a.append( (S[i] + 2*i + 2) / (i**2 + 7*i + 6) )
    if i % print_precision == 0 or i == 999998 or i == 999999:
        print(i, a[i])

\end{lstlisting}

\begin{center}
\begin{tabular}{||c | c | c | c||} 
 \hline
 \( n \) & \( p_n \) & \( q_n \) & \( s_n \) \\ [0.5ex] 
 \hline\hline
 2 & 0.5 & 0.5 & 1.0 \\ 
 \hline
 3 & 0.6666666666666666 & 0.5555555555555555 & 3.0 \\ 
 \hline
 4 & 0.5 & 0.625 & 5.0 \\ 
 \hline
 5 & 0.5333333333333334 & 0.6000000000000001 & 7.666666666666667 \\ 
 \hline
 6 & 0.5833333333333335 & 0.5925925925925926 & 11.166666666666668 \\ 
 \hline
 7 & 0.5714285714285714 & 0.5986394557823129 & 15.166666666666668 \\ 
 \hline
 8 & 0.5694444444444446 & 0.5989583333333334 & 19.722222222222225 \\ 
 \hline
 9 & 0.5767195767195766 & 0.598079561042524 & 24.912698412698415 \\ 
 \hline
 10 & 0.5791666666666668 & 0.5982539682539684 & 30.704365079365083 \\ 
 \hline
 11 & 0.5802469135802468 & 0.5984192575101667 & 37.0870811287478 \\ 
 \hline
 12 & 0.5818783068783068 & 0.5984316823437194 & 44.06962081128748 \\ 
 \hline
 13 & 0.5832778332778332 & 0.5984570510211537 & 51.65223264389931 \\ 
 \hline
 14 & 0.584370013437474 & 0.5984921698357071 & 59.833412832023946 \\ 
 \hline
 15 & 0.5853294442183335 & 0.5985192251735463 & 68.61335449529895 \\ 
 \hline
 16 & 0.586180648606244 & 0.598541831994523 & 77.99224487299885 \\ 
 \hline
 17 & 0.5869287280943053 & 0.5985622482560474 & 87.97003325060204 \\ 
 \hline
 18 & 0.5875927395506871 & 0.5985804521642101 & 98.54670256251441 \\ 
 \hline
 19 & 0.5881872747554112 & 0.5985966900970326 & 109.72226078286722 \\ 
 \hline
 20 & 0.5887224069477888 & 0.5986113039143363 & 121.496708921823 \\ 
 \hline
 21 & 0.5892065291354107 & 0.5986245302577006 & 133.87004603366663 \\ 
 \hline
 22 & 0.5896466399221238 & 0.5986365538581266 & 146.84227211195335 \\ 
 \hline
 23 & 0.5900484841483349 & 0.5986475316141904 & 160.41338724736505 \\ 
 \hline
 24 & 0.5904168410366163 & 0.5986575946089063 & 174.58339143224384 \\ 
 \hline
\end{tabular}
\end{center}

From the output, it can be seen that, for $9 \leq n \leq 24,$ both of the sequences $(p_n)$ and $(q_n)$ are increasing, as required in the proof of Proposition 3.4.

We also calculated the values of $a_n$ up to $n = 10^6,$ sufficient to find $a$ to 13 decimal places.

\begin{center}
\begin{tabular}{||c c ||} 
 \hline
 n & $a_n$ \\ [0.5ex] 
 \hline\hline
 100000 & 0.29944452190971244 \\
\hline
200000 & 0.2994445219097123 \\
\hline
300000 & 0.2994445219097122 \\
\hline
400000 & 0.29944452190971216 \\
\hline
500000 & 0.29944452190971216 \\
\hline
600000 & 0.29944452190971216 \\
\hline
700000 & 0.2994445219097123 \\
\hline
800000 & 0.29944452190971255 \\
\hline
900000 & 0.29944452190971266 \\
\hline
999998 & 0.29944452190971205 \\
\hline
999999 & 0.29944452190971205 \\
\hline
1000000 & 0.29944452190971205 \\ [1ex] 
\hline
\end{tabular}
\end{center}

The value of $a$ to seventeen decimal places is $0.29944452190971205,$ and therefore the limiting winning probability for $X$ is approximately $0.598889043819424 ~(2a).$

\textbf{Base case calculations for the best strategy for $\bm{X,}$ for low n.}

A Python algorithm was implemented to calculate the best strategy for $X,$ inductively with $n.$ The winning probabilities for $X$ for all of $MS(P_1), MS(P_2), ..., MS(P_n)$ with $X$ playing first and second that have previously been calculated are used to calculate the best strategy and winning probability for $X$ in $MS(P_{n+1}),$ by trying and comparing all possible strategies. The program is given below.

\begin{lstlisting}[breaklines]
Rf_probs = ["undefined",1]
Xf_probs = ["undefined",0]
Xf_strategies = ["undefined",[1]]

N = 30

for n in range(2,N+1):

    ### R first.

    R_total = 0

    if n % 2 == 0:
        for i in range(1,int(n/2)+1):
            if i > 1:
                R_total = R_total + (2/n) * ( ((i-1)/n)*(1 - Xf_probs[i-1]) + ((n-i)/n)*(1 - Xf_probs[n-i]) )
            else:
                R_total = R_total + (2/n) * ( ((n-i)/n)*(1 - Xf_probs[n-i]) )
    else:
        for i in range(1,int((n+1)/2)):
            if i > 1:
                R_total = R_total + (2/n) * ( ((i-1)/n)*(1 - Xf_probs[i-1]) + ((n-i)/n)*(1 - Xf_probs[n-i]) )
            else:
                R_total = R_total + (2/n) * ( ((n-i)/n)*(1 - Xf_probs[n-i]) )
        i = (n+1)/2
        R_total = R_total + (1/n) * ((n-1)/n) * (1 - Xf_probs[int((n-1)/2)])

    Rf_probs.append(1 - R_total)

    ### X first.

    X_best_prob = 0
    X_best_choice = []

    if n % 2 == 0:
        for i in range(1,int(n/2)+1):
            if i > 1:
                X_current = ( ((i-1)/n) * Rf_probs[i-1] +  ((n-i)/n) * Rf_probs[n-i] )
            else:
                X_current = ( ((n-i)/n) * Rf_probs[n-i] )
            if X_current >= X_best_prob:
                if X_current == X_best_prob:
                    X_best_choice.append(i)
                else:
                    X_best_choice = [i]
                X_best_prob = X_current
    else:
        for i in range(1,int((n+3)/2)):
            if i > 1:
                X_current = ( ((i-1)/n) * Rf_probs[i-1] +  ((n-i)/n) * Rf_probs[n-i] )
            else:
                X_current = ( ((n-i)/n) * Rf_probs[n-i] )
            if X_current >= X_best_prob:
                if X_current == X_best_prob:
                    X_best_choice.append(i)
                else:
                    X_best_choice = [i]
                X_best_prob = X_current

    Xf_probs.append(X_best_prob)
    Xf_strategies.append(X_best_choice)
\end{lstlisting}

The output lists contain the following information:
\begin{enumerate}
\item \verb|Rf_probs| gives the probabilities, starting from $n=1$ at (zero-indexed) position $1,$ that $X$ wins if $R$ plays first in $MS(P_n);$
\item \verb|Xf_probs| gives the probabilities, starting from $n=1$ at (zero-indexed) position $1,$ that $X$ wins if $X$ plays first in $MS(P_n);$
\item \verb|Xf_strategies| gives the best number(s) to choose for $X,$ starting from $n=1$ at (zero-indexed) position $1,$ in $MS(P_n).$ When $n$ is even, only the first half of the numbers ($1, 2, ..., n/2$) are considered, and similarly, when $n$ is odd, only the numbers up to the median are considered ($1, 2, ..., (n+1)/2),$ as the remaining choices have symmetric outcomes. If there were multiple equally good choices among these, all of them would be included in the output.
\end{enumerate}

The calculations are done for $n$ up to the specified value of $N$ at the start of the code.

Output with $N=23$ is shown below.

\begin{lstlisting}[breaklines]
['undefined', [1], [1], [2], [2], [2], [2], [2], [2], [2], [2], [2], [2], [2], [2], [2], [2], [2], [2], [2], [2], [2], [2], [2]]
\end{lstlisting}

Note that the result 'undefined' represents the case $n=0.$
This output shows that, when $3 \leq n \leq 23,$ choosing $2$ or $n-1$ is strictly better than any other choice for $X,$ as required for the base steps in the proof of Proposition 3.5. 
\linebreak


\begin{thebibliography}{9}

\bibitem{treesearch}
Boppana, R. B. and Lewis, J. B. \emph{The tree search game for two players.} Australasian Journal of Combinatorics, 82(2) (2022), pp. 119–145.

\bibitem{meir-moon}
Meir, A. and Moon, J. W. \emph{Cutting down random trees.} Journal of the Australian Mathematical Society, 11 (1970), 313–324.

\bibitem{markov-chainsaw}
Addario-Berry, L., Broutin, N. and Holmgren, C. \emph{Cutting down trees with a Markov chainsaw.} Annals of Applied Probability, 24 (2014), no. 6, 2297–2339.

\bibitem{random-cuttings}
Kuba, M. and Panholzer, A. \emph{Isolating a leaf in rooted trees via random cuttings.} Annals of Combinatorics, 12 (2008), 81–99.

\bibitem{novelgames}
Novel Games, 2015. [online] https://www.novelgames.com/en/guesswhat/

\bibitem{hypergeometric}
Petkovšek, M., 1992. \emph{Hypergeometric solutions of linear recurrences with polynomial coefficients.} Journal of symbolic computation, 14(2-3), pp.243-264.

\bibitem{spider}
Bahl, P., Lake, S. and Wertheim, A. \emph{Gracefulness of families of spiders.} Involve, 3(2010), 241-247.

\bibitem{nim}
Bouton, C. L. \emph{Nim, a game with a complete mathematical theory.} Annals of Mathematics, 3 (1901), 35–39.

\bibitem{caterpillar}
Harary, F. and Schwenk, A. J. \emph{The Number of Caterpillars.} Discrete Mathematics, 6 (1973), 359-365.

\bibitem{gold}
Seacrest, D. E. and Seacrest, T. \emph{Grabbing the gold.} Discrete Mathematics, 312 (2012), 1804–1806.

\bibitem{graph_binary_search}
Emamjomeh-Zadeh, E., Kempe, D. and Singhal, V. \emph{Deterministic and Probabilistic Binary Search in Graphs.} 48th Annual ACM Symposium on Theory of Computing (2016), pp. 519–532.

\bibitem{sage}
The Sage Developers. SageMath, the Sage Mathematics Software System (Version 9.1), 2020.

\end{thebibliography}
\end{document}